\documentclass[12pt]{amsart}
\usepackage{amscd,amssymb,latexsym}
\usepackage{fullpage}
\newcommand{\Mdef}[2]{\newcommand{#1}{\relax \ifmmode #2 \else $#2$\fi}}

\newcommand{\sm }{\wedge}

\newcommand{\map}{\mathrm{map}}

\Mdef{\bhom}{\mathbf{\hat{H}om}}

\Mdef{\Mod}{\mathrm{mod}}

\newcommand{\st}{\; | \;}

\newtheorem{thm}{Theorem}[section]
\newtheorem{lemma}[thm]{Lemma}

\newtheorem{cor}[thm]{Corollary}

\theoremstyle{definition}

\newcommand{\qqed}{\qed \\[1ex]}
\renewenvironment{proof}[1][\hspace*{-.8ex}]{\noindent {\bf Proof #1:\;}}{\qqed}

\Mdef{\PH} {\Phi^H}
\Mdef{\PK} {\Phi^K}
\Mdef{\PL} {\Phi^L}
\Mdef{\PT} {\Phi^{\T}}

\Mdef{\ef}{E{\cF}_+}
\Mdef{\etf}{\widetilde{E}{\cF}}
\Mdef{\eg}{E{G}_+}
\Mdef{\etg}{\tilde{E}{G}}

\Mdef{\infl}{\mathrm{inf}}
\Mdef{\defl}{\mathrm{def}}
\Mdef{\res}{\mathrm{res}}
\Mdef{\ind}{\mathrm{ind}}
\Mdef{\coind}{\mathrm{coind}}

\Mdef{\univ}{\mathcal{U}}

\Mdef{\Fp}{\mathbb{F}_p}
\Mdef{\Zpinfty}{\Z /p^{\infty}}
\Mdef{\Zpadic}{\Z_p^{\wedge}}

\newcommand{\bi}{\begin{itemize}}
\newcommand{\be}{\begin{enumerate}}
\newcommand{\bc}{\begin{center}}
\newcommand{\bd}{\begin{description}}
\newcommand{\ei}{\end{itemize}}
\newcommand{\ee}{\end{enumerate}}
\newcommand{\ec}{\end{center}}
\newcommand{\ed}{\end{description}}

\newcommand{\lra}{\longrightarrow}

\Mdef{\we}{\mathbf{we}}
\Mdef{\fib}{\mathbf{fib}}
\Mdef{\cof}{\mathbf{cof}}
\Mdef{\BI}{\mathcal{BI}}

\Mdef{\B}{\mathbb{B}}
\Mdef{\C}{\mathbb{C}}
\Mdef{\D}{\mathbb{D}}
\Mdef{\E}{\mathbb{E}}
\Mdef{\T}{\mathbb{T}}
\Mdef{\F}{\mathbb{F}}
\Mdef{\G}{\mathbb{G}}
\Mdef{\I}{\mathbb{I}}
\Mdef{\N}{\mathbb{N}}
\Mdef{\Q}{\mathbb{Q}}
\Mdef{\R}{\mathbb{R}}
\Mdef{\bbS}{\mathbb{S}}
\Mdef{\Z}{\mathbb{Z}}

\Mdef{\bA}{\mathbb{A}}
\Mdef{\bB}{\mathbb{B}}
\Mdef{\bC}{\mathbb{C}}
\Mdef{\bD}{\mathbb{D}}
\Mdef{\bE}{\mathbb{E}}
\Mdef{\bF}{\mathbb{F}}
\Mdef{\bG}{\mathbb{G}}
\Mdef{\bH}{\mathbb{H}}
\Mdef{\bI}{\mathbb{I}}
\Mdef{\bJ}{\mathbb{J}}
\Mdef{\bK}{\mathbb{K}}
\Mdef{\bL}{\mathbb{L}}
\Mdef{\bM}{\mathbb{M}}
\Mdef{\bN}{\mathbb{N}}
\Mdef{\bO}{\mathbb{O}}
\Mdef{\bP}{\mathbb{P}}
\Mdef{\bQ}{\mathbb{Q}}
\Mdef{\bR}{\mathbb{R}}
\Mdef{\bS}{\mathbb{S}}
\Mdef{\bT}{\mathbb{T}}
\Mdef{\bU}{\mathbb{U}}
\Mdef{\bV}{\mathbb{V}}
\Mdef{\bW}{\mathbb{W}}
\Mdef{\bX}{\mathbb{X}}
\Mdef{\bY}{\mathbb{Y}}
\Mdef{\bZ}{\mathbb{Z}}

\Mdef{\cA}{\mathcal{A}}
\Mdef{\cB}{\mathcal{B}}
\Mdef{\cC}{\mathcal{C}}
\Mdef{\mcD}{\mathcal{D}} 
\Mdef{\cE}{\mathcal{E}}
\Mdef{\cF}{\mathcal{F}}
\Mdef{\cG}{\mathcal{G}}
\Mdef{\mcH}{\mathcal{H}} 
\Mdef{\cI}{\mathcal{I}}
\Mdef{\cJ}{\mathcal{J}}
\Mdef{\cK}{\mathcal{K}}
\Mdef{\mcL}{\mathcal{L}}

\Mdef{\cM}{\mathcal{M}}
\Mdef{\cN}{\mathcal{N}}
\Mdef{\cO}{\mathcal{O}}
\Mdef{\cP}{\mathcal{P}}
\Mdef{\cQ}{\mathcal{Q}}
\Mdef{\mcR}{\mathcal{R}}
\Mdef{\cS}{\mathcal{S}}
\Mdef{\cT}{\mathcal{T}}
\Mdef{\cU}{\mathcal{U}}
\Mdef{\cV}{\mathcal{V}}
\Mdef{\cW}{\mathcal{W}}
\Mdef{\cX}{\mathcal{X}}
\Mdef{\cY}{\mathcal{Y}}
\Mdef{\cZ}{\mathcal{Z}}

\Mdef{\ca}{\mathcal{a}}
\Mdef{\ct}{\mathcal{t}}

\Mdef{\At}{\tilde{A}}
\Mdef{\Bt}{\tilde{B}}
\Mdef{\Ct}{\tilde{C}}
\Mdef{\Ht}{\tilde{H}}
\Mdef{\Kt}{\tilde{K}}
\Mdef{\Lt}{\tilde{L}}
\Mdef{\Mt}{\tilde{M}}
\Mdef{\Nt}{\tilde{N}}
\Mdef{\Pt}{\tilde{P}}

\Mdef{\tA}{\tilde{A}}
\Mdef{\tB}{\tilde{B}}
\Mdef{\tC}{\tilde{C}}
\Mdef{\tE}{\tilde{E}}
\Mdef{\tH}{\tilde{H}}
\Mdef{\tK}{\tilde{K}}
\Mdef{\tL}{\tilde{L}}
\Mdef{\tM}{\tilde{M}}
\Mdef{\tN}{\tilde{N}}
\Mdef{\tP}{\tilde{P}}

\Mdef{\ft}{\tilde{f}}
\Mdef{\xt}{\tilde{x}}
\Mdef{\yt}{\tilde{y}}

\Mdef{\Ab}{\overline{A}}
\Mdef{\Bb}{\overline{B}}
\Mdef{\Cb}{\overline{C}}
\Mdef{\Db}{\overline{D}}
\Mdef{\Eb}{\overline{E}}
\Mdef{\Fb}{\overline{F}}
\Mdef{\Gb}{\overline{G}}
\Mdef{\Hb}{\overline{H}}
\Mdef{\Ib}{\overline{I}}
\Mdef{\Jb}{\overline{J}}
\Mdef{\Kb}{\overline{K}}
\Mdef{\Lb}{\overline{L}}
\Mdef{\Mb}{\overline{M}}
\Mdef{\Nb}{\overline{N}}
\Mdef{\Ob}{\overline{O}}
\Mdef{\Pb}{\overline{P}}
\Mdef{\Qb}{\overline{Q}}
\Mdef{\Rb}{\overline{R}}
\Mdef{\Sb}{\overline{S}}
\Mdef{\Tb}{\overline{T}}
\Mdef{\Ub}{\overline{U}}
\Mdef{\Vb}{\overline{V}}
\Mdef{\Wb}{\overline{W}}
\Mdef{\Xb}{\overline{X}}
\Mdef{\Yb}{\overline{Y}}
\Mdef{\Zb}{\overline{Z}}

\Mdef{\db}{\overline{d}}
\Mdef{\hb}{\overline{h}}
\Mdef{\qb}{\overline{q}}
\Mdef{\rb}{\overline{r}}
\Mdef{\tb}{\overline{t}}
\Mdef{\ub}{\overline{u}}
\Mdef{\vb}{\overline{v}}

\Mdef{\hc}{\hat{c}}
\Mdef{\he}{\hat{e}}
\Mdef{\hf}{\hat{f}}
\Mdef{\hA}{\hat{A}}
\Mdef{\hH}{\hat{H}}
\Mdef{\hJ}{\hat{J}}
\Mdef{\hM}{\hat{M}}
\Mdef{\hP}{\hat{P}}
\Mdef{\hQ}{\hat{Q}}

\Mdef{\thetab}{\overline{\theta}}
\Mdef{\phib}{\overline{\phi}}

\Mdef{\uA}{\underline{A}}
\Mdef{\uB}{\underline{B}}
\Mdef{\uC}{\underline{C}}
\Mdef{\uD}{\underline{D}}

\Mdef{\bolda}{\mathbf{a}}
\Mdef{\boldb}{\mathbf{b}}
\Mdef{\bfD}{\mathbf{D}}

\Mdef{\fm}{\frak{m}}
\Mdef{\fp}{\frak{p}}

\Mdef{\eps}{\epsilon}

\newcommand{\up}{\mathsf{V}}

\newcommand{\EEinftyG}{{\mathbb E}_{\infty}^G}
\newcommand{\EEinfty}{{\mathbb E}_{\infty}}
\newcommand{\EinftyG}{E_{\infty}^G}
\newcommand{\Einfty}{E_{\infty}}
\newcommand{\Ninfty}{N_{\infty}}
\newcommand{\GS}{G\times \Sigma_n}
\newcommand{\Et}{\tilde{E}}
\input{xypic}

\begin{document}
\title{Couniversal spaces which are  equivariantly commutative ring spectra}

\author{J.P.C.Greenlees}
\address{School of Mathematics and Statistics, Hicks Building, 
Sheffield S3 7RH. UK.}
\email{j.greenlees@sheffield.ac.uk}
\date{}

\begin{abstract}
We identify  which  which couniversal spaces have
suspension spectra equivalent to commutative orthogonal ring
$G$-spectra for a compact Lie group $G$. These are precisely those whose cofamily is closed under
passage to finite index subgroups. Equivalently these are the 
couniversal spaces admitting an action of an 
 $\EinftyG$-operad.   
\end{abstract}

\thanks{I am grateful to M.Hill and M.Kedziorek for the conversation
  at EuroTalbot17 when we observed that we knew of no obstruction to
  Corollary \ref{cor:main}.  }
\maketitle

\tableofcontents

\section{Introduction}
For a compact Lie group $G$, Theorem \ref{thm:main} shows that a number of simple
$G$-equivariant homotopy types
have suspension spectra which are commutative orthogonal ring
$G$-spectra.  
Because equivariant commutativity implies a large amount of additional
structure, including norm maps, this has significant implications. 

These homotopy types are naturally used for
isotropic decompositions of the sphere, and as such they  play a
significant role in understanding the structure of $G$-equivariant
spectra where $G$ is a torus in \cite{tnqcore}. That analysis involves
constructing the model category of
rational $G$-spectra for a torus $G$ from a diagram of much simpler
model categories. The simplest way to do this is to construct the
simpler model categories as categories of modules over commutative ring $G$-spectra,
and for this diagram to arise from a diagram of commutative ring $G$-spectra. 

The homotopy types of the ring $G$-spectra are apparent from the construction, and it remains
to show that they are indeed commutative ring $G$-spectra. If the
ambient category of $G$-spectra is the category of orthogonal spectra,
the commutative monoids admit multiplicative norm maps, which is a
substantial restriction on the homotopy type. Accordingly, \cite{tnqcore} works instead with
the Blumberg-Hill category of orthogonal  $\mathcal{L}$-spectra \cite{BlumbergHillL1}, where 
 many more $G$-spectra admit the structure of commutative rings. 
 
The motivating application of the present note is to show that in fact the ring
spectra required in the construction of  \cite{tnqcore} can be represented by commutative
rings in the category orthogonal $G$-spectra. It follows that the the argument of
\cite{tnqcore} can be conducted directly in the category of orthogonal
$G$-spectra rather than in the more elaborate category of spectra with an
$\mathcal{L}$-action.

\section{Operadic preliminaries}

There are rare examples of spectra which are obviously strictly 
commutative rings, but it is much more usual to show that a spectrum 
admits the action of a suitable operad, and then use general results 
to show this means the homotopy type is represented by a ring 
spectrum. 

\subsection{$\Ninfty$-operads}
In the equivariant world there is a range of essentially different
operads governing commutative ring spectra: these are the
$\Ninfty$-operads of Blumberg-Hill \cite{BlumbergHillNorms}. These are operads $\cO$
in $G$-spaces whose $n$-th term $\cO(n)$ is a universal space for a
family  $\cF \cO (n)$ of subgroups of $\GS$: it is essential that $\cO
(n)$ is $G$-fixed and $\Sigma_n$-free, but within that class there is
a wide range of options. We need only discuss the two extreme types of $\Ninfty$ $G$-operads. 

At one extreme we have the non-equivariant $\Einfty$-operads, which
are as free as possible whilst being  $G$-fixed. Equivalently, 
the $n$-th term is the universal space for the family 
$$\cF (n)=\{ H\times 1 \subseteq \GS \st H\subseteq G\}. $$ 
There are of course many $\Einfty$-operads, and we write 
$\EEinfty$ for a chosen one. For example we might use 
 the linear isometries 
operad on a $G$-fixed universe, but we will use  no special  properties of the 
operad. 

At the other extreme we have the  $\EinftyG$-operads which are as
fixed as possible whilst their $n$th term is $\Sigma_n$-free, so their
$n$-th term is  a universal space for the family 
$$\cF_G(n) =\{ \Gamma \st \Gamma \cap \Sigma_n=1\}.$$ 
There are of course many $\EinftyG$-operads, and we
write $\EEinftyG$ for a chosen one. For example we might use
 the linear isometries 
operad on a complete $G$-universe, but we will use no special properties of the 
operad. 
We pause to recall that  if $\Gamma \cap \Sigma_n=1$ then $\Gamma$  is a `graph subgroup' in the sense 
that  we have $\Gamma =\Gamma(L, \alpha)$ for some subgroup 
$L$ of $G$ and some homomorphism $\alpha : L \lra \Sigma_n$, where 
$\Gamma (L, \alpha)=\{ (x , \alpha (x))\st x \in L\}$. 

\subsection{Commutative monoids and $\EinftyG$-operads}
The relevance of $\EinftyG$-operads is the connection to the standard
symmetric monoidal product of spectra. 

\begin{lemma}
The commutative monoids in the category of orthogonal $G$ spectra are
the   $\EinftyG$-algebras.
\end{lemma}

\begin{proof}
This uses the traditional argument of
\cite[15.5]{MMSS}, using \cite[B.117]{HHR}, which in turn corrects 
\cite[III.8.4]{MMorthogonal}. We note that the statement in \cite{HHR} is only
given for finite groups, but the argument applies as written to arbitrary
compact Lie groups, giving the full replacement for the statement in  \cite{MMorthogonal}. 
\end{proof}

\subsection{Endomorphism operads}
The other piece of standard material is to consider the endomorphism
operad $\cE_Y$ on a based space  $Y$, defined by
$$\cE_Y(n)=Map_*(Y^{\sm n }, Y). $$
We automatically find $Y$ is an $\cE_Y$-algebra. Equally, if $Y$ is a
based $G$-space $\cE_Y$ is an operad in $G$-spaces and $Y$ is an algebra
over it.

\section{McClure's argument}
McClure \cite{McClureTate} argued as follows to construct an $\Einfty$-operad acting on
$\Et G$. 

First we consider the endomorphism operad $\cE_{\Et G}$, and then note that passage to fixed
points gives a map 
$$\phi(n): \cE_{\Et G}(n)^G=Map_*^G(\Et G^{\sm n}, \Et G)\lra Map_*(S^0,
S^0). $$
We write 
$$D_{McC}(n)=\phi(n)^{-1}(id), $$
and note that this is also an operad acting on $\Et G$. Because $\phi
(n)$ is a weak equivalence $D_{McC}(n)$ is contractible, so that 
$\EEinfty \times D_{McC}$ is an $\Einfty$-operad acting on $\Et G$ as required. 

\section{Generalizing McClure's argument}

\subsection{Couniversal spaces}
Given a group $G$ and a family $\cF$ of subgroups of $G$, 
we  say that $\Et \cF=S^0*E\cF$ is {\em the couniversal space} for the
complementary cofamily 
$All \setminus \cF$.  Simplifying notation, for a  cofamily $\cC$,
we write simply 
$$E\cC =\Et (\cC^c). $$
This has two essential features: it has geometric isotropy $\cC$, and
$(E\cC)^H=S^0$ whenever $H\in \cC$.

\subsection{The endomorphism operad of a cofamily}
\label{subsec:counivoperad}
We consider the endomorphism operad of $E\cC$:  
$$\cE_{E\cC}(n) = Map_*(E\cC^{\sm n},E\cC). $$
The following partial information about the homotopy type of this
space will be useful later.

\begin{lemma}
\label{lem:easyendo}
Given cofamilies $\cC$ and $\mcD$ the space 
$$\map_* (E\cC ,E\mcD )$$
has the following properties
\begin{itemize}
\item It is $H$-contractible if $H\not \in \cC \cap \mcD$
\item It is $H$-couniversal if no subgroup of $H$ lies in $\mcD
  \setminus \cC$
\end{itemize}
\end{lemma}

\begin{proof}
It is clear that if $H$ is not in $\cC \cap \mcD$ then $\map_*(E\cC, 
E\mcD)$ is $H$-contractible, since one or other of the spaces is.

If $H \in \cC\cap \mcD$ we wish to argue that the map 
$$\map_*(E\cC , E\mcD)^H \lra \map_*(S^0, S^0)=S^0$$
is an equivalence. In other words, that any $H$-map $f: E\cC \lra 
E\mcD$ is determined by the map from $S^0\lra E\mcD$. The obstruction
to extension and uniqueness lie in $[E\cC^c_+\sm S^k, E\mcD]^H$,
which vanishes unless $H$ has a subgroup $K \in \mcD \setminus \cC$.
\end{proof}

\subsection{The couniversal operad of a cofamily}

There is a $\GS$-map $i_n: S^0=(S^0)^{\sm n} \lra (E\cC)^{\sm n}$ inducing
a $\GS$-map 
$$i_n^*: \cE_{E\cC} (n)  = Map_*(E\cC^{\sm n},E\cC)\lra
Map_*(S^0,E\cC)=E\cC . $$
We take
$$D\cC (n)=(i_n^*)^{-1}(i_1).   $$
We note that when $\cC={\mathcal{NT}}$ consists of the non-trivial
subgroups the fixed point set $D{\mathcal{NT}}^G=D_{McC}$ is
McClure's operad. 

\begin{lemma}
\label{lem:DCactsonEC}
$D\cC$ is an operad acting on $E\cC$. \qqed
\end{lemma}


Using this,  we will show that for suitable cofamilies $\cC$, the space $E\cC$ 
is an algebra over an $\Ninfty$-operad with more highly structured
algebras than $\EEinfty$.



\subsection{Permutation powers and cofamilies}
Let us think of the symmetric group $\Sigma_n$ as the permutations of
$\{ 1, 2, \ldots, n\}$. We consider the group $\GS$ and 
let $p:\GS \lra \Sigma_n$ and $\pi: \GS \lra G$ be the projections. 

If $\cC$ is a cofamily of subgroups of $G$, we view $E\cC$  as a trivial
$\Sigma_{n-1}$-space and form the $n$th smash power  $(E\cC)^{\sm n}$
and view it as a $\GS$-space.

\begin{lemma}
The $\GS$-space $E\cC^{\sm n}$ is couniversal. 
\end{lemma}

\begin{proof}
Consider any $G$-space $X$ and form the $\GS$-space $X^{\sm n}$.
We will consider fixed points under  a subgroup
 $\Delta \subseteq \GS$. 

Consider the orbits $o_1, \ldots , o_s$ of
 $\{1, \ldots , n\}$ under $p(\Delta)$, and choose orbit
representatives $d_i\in o_i$. Now write
$\Delta_i=p^{-1}((\Sigma_n)_{d_i})\cap \Delta$ for the subgroup of 
$\Delta$ fixing $d_i$. 

We then see that there is a homeomorphism 
$$h: \bigwedge_{i=1}^s X^{\pi (\Delta_i)}\stackrel{\cong}\lra (X^{\wedge
  n})^{\Delta}. $$
The $i$th factor in the domain gives the $d_i$th coordinate in
$X^{\wedge n}$ and hence determines the coordinates in $o_i$. 
More precisely, if $m\in o_i$ we may choose $\delta \in \Delta$ 
with $p(\delta)(d_i)=m$, and then 
$$h(x_1\sm \ldots \sm x_s)_m=\pi (\delta ) x_i. $$
Since $x_i$ is fixed by $\Delta_i$ this is independent of the choice
of $\delta$. The verification that $h$ is a homeomorphism is
straightforward. 

Applying this to $X=E\cC$ we see that $X^{\Delta}$ is always either
$S^0$ or contractible. The collection of subgroups for which it is
$S^0$ is obviously a cofamily. 





\end{proof}

If we write $C(\cC, n)$ for the geometric
isotropy of $E\cC^{\sm n}$, then by the lemma $E\cC^{\sm n}\simeq EC(\cC,n)$.

\begin{lemma}
\label{lem:CCnC1}
$$C(\cC,n) \subseteq \pi^* \cC$$
\end{lemma}

\begin{proof}
We show that if $\Delta$ is not in the right hand side it is not in
the left hand side. 

If $\pi (\Delta) $ does not lie in $\cC$,  then $(E\cC)^{\pi
  (\Delta)}\neq S^0$. Suppose then that $x=x(1)\in E\cC \setminus S^0$
  is a non-trivial element of $\cC$ fixed by $\pi (\Delta)$. Now write $x(i)=\sigma^i x(1)$ where
  $\sigma =(123\cdots n)$.  We then have $x =x(1)\sm w(2)\sm \cdots
  \sm w(n)$ fixed by $\pi (\Delta )\times \Sigma_n$ and hence by its
  subgroup $\Delta$. Hence $\Delta \not \in C(\cC , n)$. 
\end{proof}

\begin{lemma}
\label{lem:CC1Cn}
If $\cC$ is closed under passage to finite index subgroups then 
$$\pi^*\cC \cap \cF_G(n) \subseteq C(\cC ,n)$$
\end{lemma}

\begin{proof}
Suppose $\Delta\subseteq \GS$ lies in the intersection, which is to
say  $L:=\pi (\Delta )\in \cC$, and  $\Delta =\Gamma (L, \alpha)$ is a graph
subgroup. We will show that $\Delta \in C(\cC,n)$. Since $C(\cC , n)$
is a cofamily, it suffices to show that the subgroup
$\Delta'=\Gamma(L_e, \alpha|_{L_e})$ lies in $C(\cC,n)$, where $L_e$ is the
identity component of $L$. 

However, since $\Sigma_n$ is discrete $\alpha|_{L_e}$ is trivial, 
so that $\Delta'=\Gamma (L_e, const)=L_e$. However $L=\pi(\Delta)$ lies in
$\cC$, so its finite index subgroup $L_e$ also lies in
$\cC$ by hypothesis:
$$(E\cC^{\sm n})^{\Delta}\subseteq (E\cC^{\sm n})^{\Delta'}=
(E\cC^{L_e})^{\sm n})=(S^0)^{\sm n}=S^0. $$
Hence $(E\cC^{\sm n})^{\Delta}=S^0$ as required. 
\end{proof}

\begin{lemma}
\label{lem:inFGn}
The  map 
$$i_n^*: \cE_{E\cC}(n)=\map_*(E\cC^n , E\cC)\lra \map_*(S^0, 
E\cC)=E\cC$$
is an $\cF_G(n)$-equivalence.
\end{lemma}

\begin{proof}
We observe that by Lemmas \ref{lem:CCnC1} and \ref{lem:CC1Cn},  if $H\in \cF_G(n)$ then $C(\cC,
n)|_H=\pi_*\cC|_H$. The result follows from Lemma \ref{lem:easyendo}. 
\end{proof}

\subsection{McClure's argument extended}

We now apply the above to the operad $D\cC$ of Subsection
\ref{subsec:counivoperad}. 

\begin{thm}
\label{thm:main}
If $\cC$ is a cofamily then the  space $E\cC$ is an $\EinftyG$-algebra
if and only if $\cC$ is closed under passage to finite index
subgroups. 
\end{thm}

\begin{proof}
If there is a finite index inclusion  $K\subseteq H$ of subgroups with
$H \in \cC$ and $K\not\in \cC$, then the assumption that $E\cC$ is
$\EinftyG$ leads to a contradiction. Indeed $\pi^K_0(E\cC)=0$ so that
$1=0$ in that ring. On the other hand, by Segal-tom Dieck splitting, 
$\pi^H_0(E\cC)\neq 0$ so that $1\neq 0$ in $\pi^H_0(E\cC)$. 
The existence of a norm map then gives a
contradiction since $\mathrm{norm}_K^H(1)=1$. 

Now suppose $\cC$ is closed under passage to finite index subgroups. 
By Lemma \ref{lem:DCactsonEC} there is an action of $D\cC$ on
$E\cC$, and hence also an action of  $\EEinftyG \times D\cC$. 
It remains to show that the $n$th term in this operad is universal for
$\cF_G(n)$. In other words, we need to show that if $\Gamma\in \cF_G(n)$
is a graph subgroup then $D\cC(n)^{\Gamma}\simeq *$.

Now by Lemma \ref{lem:inFGn},  the  map 
$$i_n^*: \cE_{E\cC}(n)=\map_*(E\cC^n , E\cC)\lra \map_*(S^0, 
E\cC)=E\cC$$
is an $\cF_G(n)$-equivalence, and hence 
$D\cC (n)$ is $\cF_G(n)$-contractible as required. 
\end{proof}

\begin{cor}
\label{cor:main}
If $G$ is a torus and $K$ is a connected subgroup then $S^{\infty
  V(K)}=\bigcup_{V^K=0}S^V$ is an $\EinftyG$-algebra. 
\end{cor}

\begin{proof}
The space $S^{\infty V(K)}$ is couniversal for the cofamily $\up
(K)=\{H \st H\supseteq K\}$ of subgroups containing $K$. Since $K$ is
connected $\up (K)$ is closed under passage to finite index
subgroups. 
\end{proof}


\begin{thebibliography}{1}

\bibitem{BlumbergHillL1}
A.~J. Blumberg and M.~A. Hill.
\newblock {$G$}-symmetric monoidal categories of modules over equivariant
  commutative ring spectra.
\newblock In preparation.

\bibitem{BlumbergHillNorms}
A.~J. Blumberg and M.~A. Hill.
\newblock Operadic multiplications in equivariant spectra, norms, and
  transfers.
\newblock {\em Adv. Math.}, 285:658--708, 2015.

\bibitem{tnqcore}
J.~P.~C. Greenlees and B.~Shipley.
\newblock An algebraic model for rational torus-equivariant spectra.
\newblock {\tt arXiv: 1101:2511}.

\bibitem{HHR}
M.~A. Hill, M.~J. Hopkins, and D.~C. Ravenel.
\newblock On the nonexistence of elements of {K}ervaire invariant one.
\newblock {\em Ann. of Math. (2)}, 184(1):1--262, 2016.

\bibitem{MMorthogonal}
M.~A. Mandell and J.~P. May.
\newblock Equivariant orthogonal spectra and {$S$}-modules.
\newblock {\em Mem. Amer. Math. Soc.}, 159(755):x+108, 2002.

\bibitem{MMSS}
M.~A. Mandell, J.~P. May, S.~Schwede, and B.~Shipley.
\newblock Model categories of diagram spectra.
\newblock {\em Proc. London Math. Soc. (3)}, 82(2):441--512, 2001.

\bibitem{McClureTate}
J.~E. McClure.
\newblock {$E_\infty$}-ring structures for {T}ate spectra.
\newblock {\em Proc. Amer. Math. Soc.}, 124(6):1917--1922, 1996.

\end{thebibliography}
\end{document}